\documentclass[11pt]{35}
\usepackage{amsthm, amsmath, amssymb, amsbsy, amsfonts, latexsym, euscript}
\usepackage{calrsfs} %kalligraphiya
\usepackage{graphicx}

\DeclareMathAlphabet{\varmathbb}{U}{pxsyb}{m}{n}

\def\leq{\leqslant}

\def\geq{\geqslant}
\def\phi{\varphi}

\def\bar{\overline}

\def\kappa{\varkappa}

\newcommand{\D}{\mathrm{d}\kern0.2pt}%

\newcommand{\E}{\mathrm{e}\kern0.2pt} 
\newcommand{\ii}{\kern0.05em\mathrm{i}\kern0.05em}

\newcommand{\RR}{\mathbb{R}}%

\newtheorem{theorem}{\bf \indent Theorem}[section]
\newtheorem{proposition}{\bf \indent Proposition}[section]

\theoremstyle{remark}

\numberwithin{equation}{section}

\begin{document}

\noindent {\Large \bf The convexity of a planar domain via properties of solutions
\\[3pt] to the modified Helmholtz equation}

\vskip5mm

{\bf Nikolay Kuznetsov}

\vskip-2pt {\small Laboratory for Mathematical Modelling of Wave Phenomena}
\vskip-4pt {\small Institute for Problems in Mechanical Engineering, RAS} \vskip-4pt
{\small V.O., Bol'shoy pr. 61, St Petersburg 199178, Russia \vskip-4pt {\small
nikolay.g.kuznetsov@gmail.com}

\vskip7mm

\parbox{144mm} {\noindent A new characterization of convexity of a planar domain is
obtained. Its derivation involves two classical facts: the Varadhan's formula,
expressing the distance function with respect to the domain's boundary via
real-valued solutions of the modified Helmholtz equation, and the convexity of a
planar domain in which the distance function is superharmonic.}

\vskip6mm

{\centering \section{Introduction} }

\noindent Let a point $x = (x_1, x_2) \in \RR^2$ belong to a half-plane, say $x_2 >
0$. It is obvious that the distance from $x$ to the $x_1$-axis bounding this
half-plane\,---\,the length of the segment orthogonal to the $x_1$-axis\,---\,is
equal to $x_2$, and so is harmonic in this domain. The converse assertion is far
from being trivial to prove; it is referred to as one of ``three secrets about
harmonic functions'' in the interesting note~\cite{B}. A half-plane is the maximal,
as it were, convex domain in $\RR^2$, whereas our aim is to characterize in a new
fashion the convexity of a bounded planar domain.

In an arbitrary planar domain $D$, the distance function with respect to $\partial
D$ is
\[ d (x, \partial D) = \inf_{y \in \partial D} |x-y| \, ,
\]
where $x \in D$ and $|x-y| = \sqrt{(x_1 - y_1)^2 + (x_2 - y_2)^2}$ is the Euclidean
norm in $\RR^2$. It is known that the superharmonicity of this function in $D$
implies that the domain is convex; see Theorem~1.3 below (it was obtained by
Armitage and Kuran \cite{AK} in 1985). In the present note, we demonstrate that it
is also possible to describe the domain's convexity by means of real-valued
solutions of the modified Helmholtz equation (labelled panharmonic functions in
\cite{D}); indeed, this is achieved by imposing some conditions on them.

\vspace{2mm}

{\bf 1.1. Statement of the problem; formulation of the result.} Throughout this
paper~$D$ denotes a bounded domain of the Euclidean plane $\RR^2$. Its boundary
$\partial D$ is assumed to be sufficiently smooth, say $C^2$, and so the Dirichlet
problem for the modified Helmholtz equation
\begin{equation}
\nabla^2 v - \mu^2 v = 0 \ \ \mbox{for} \ x = (x_1, x_2) \in D, \quad \mu \in \RR_+
, \label{MHh}
\end{equation}
has a unique solution satisfying the boundary condition
\begin{equation}
v (x) = 1 \quad \mbox{for} \ x \in \partial D ; \label{bc}
\end{equation}
here and below $\nabla = (\partial_1, \partial_2)$ is the gradient operator,
$\partial_i = \partial / \partial x_i$.

It has long bean known (see \cite[p.~437]{V}) that a solution of this problem, say
$v (x, \mu)$, satisfies
\begin{equation}
0 \leq v (x, \mu) \leq 1 \ \ \mbox{for all} \ x \in \bar D \ \mbox{and any} \ \mu >
0. \label{ineq}
\end{equation}
Moreover, the left inequality is strict; see \cite[Sect.~2.2]{HMOSY} for the proof.

It occurs that problem's solutions corresponding to large values of $\mu$
characterize the domain~$D$. Namely, our aim is to prove the following

\begin{theorem}
Let $D \subset \RR^2$ be a bounded domain with smooth boundary. If the condition
\begin{equation}
|\nabla v (x, \mu)| \leq \mu v (x, \mu) , \ \ x \in D , \label{cond}
\end{equation}
holds for solutions of problem \eqref{MHh}, \eqref{bc} with all large $\mu > 0$, then
$D$ is convex.
\end{theorem}

This theorem is proved in Section 2, and here we illustrate it with two elementary
examples.

In the first one, the domain $D$ is the half-plane $x_2 > 0$, in which case it is
natural to require
\[ v (x, \mu) \to 0 \quad \mbox{as} \ x_2 \to \infty . 
\]
Duffin established that $v (x, \mu) = \E^{- \mu x_2}$ solves problem \eqref{MHh},
\eqref{bc} complemented by the latter condition; see his proof of \cite[Theorem
5]{D}. Condition \eqref{cond} is fulfilled for this function for all $\mu > 0$, and
so Theorem~1.1 implies the obvious fact that $D$ is convex. Notice that
\[ - \mu^{-1} \log v (x, \mu) = x_2 = d (x, \partial D)
\]
in agreement with the assertion of Theorem~1.2 formulated below.

In the second example, the domain $D$ is a disc, and so convex. As usual, $I_0$
stands for the modified Bessel function of the first kind (see, for example,
\cite[p.~111]{D}). If $a \in (0, 1)$, then $v (x, \mu) = a I_0 (\mu |x|)$ solves
problem \eqref{MHh}, \eqref{bc} in the disc $B_r (0) = \{ y \in \RR^2 : |y| < r \}$
provided $r > 0$ is such that $a I_0 (\mu r) = 1$.

It is clear that properties of the function $I_0$ (see \cite[sect.~9.6]{AS}) yield
condition \eqref{cond} for $v (\cdot, \mu)$; indeed, it takes the form $I_1 (\mu |x|)
< I_0 (\mu |x|)$, because
\[ |\nabla v (x, \mu)| = \mu a I_1 (\mu |x|)
\]
in view of formula \cite[9.6.27]{AS}.

\vspace{2mm}

{\bf 1.2. Background.} Here we formulate two classical results about the distance
function with respect to $\partial D$. These results are essential for proving
Theorem 1.1.

\begin{theorem}[Varadhan, \cite{V}]
Let $D \subset \RR^2$ be a bounded domain with smooth boundary. If $v (x, \mu)$ is a
solution of the Dirichlet problem \eqref{MHh}, \eqref{bc}, then 
\begin{equation}
- \mu^{-1} \log v (x, \mu) \to d (x, \partial D) \quad as \ \mu \to \infty
\label{lim}
\end{equation}
uniformly on $D$.
\end{theorem}

It is worth mentioning that the original Varadhan's result involves the equation in
$\RR^m$, $m \geq 2$, with a general elliptic operator of second order instead of the
Laplacian, and the distance is induced by a Riemannian metric derived from the
operator's coefficients. For this reason the Varadhan's proof is rather complicated.
In the recent preprint \cite{HMOSY}, signed distance functions are considered, thus
generalizing $d (\cdot, \partial D)$. For this purpose the source term is added to
equation \eqref{MHh}, which allows the authors to extend Theorem~1.2; moreover, the
rate of convergence of $- \mu^{-1} \log v (x, \mu)$ to $d (x, \partial D)$ is
evaluated in the preprint.

Furthermore, let us consider $v (x, \mu)$, satisfying equation \eqref{MHh} and the
Neumann condition
\[ \partial v / \partial n = \mu \quad \mbox{on} \ \partial D 
\]
($n$ is the exterior normal on $\partial D$), instead of the Dirichlet condition
\eqref{bc}. It is interesting to find out whether $- \mu^{-1} \log v (x, \mu)$
converges to $d (x, \partial D)$ in $D$ as $\mu \to \infty$ in this case.

Another classical result has a simple proof which is reproduced here.

\begin{theorem}[Armitage and Kuran, \cite{AK}]
If $d (\cdot, \partial D)$ is superharmonic in a domain $D \subset \RR^2$, then $D$
is convex.
\end{theorem}

\begin{proof}
Let us assume $D \subset \RR^2$ to be nonconvex and show that $d (\cdot, \partial
D)$ is not superharmonic in~$D$. According to the assumption, there exist a point
$y \in \partial D$, a closed half-plane $P$ with $y \in \partial P$ (without loss of
generality, we set $y = (0,0)$ and $P = \{ x \in \RR^2 : x_2 \geq 0 \}$), and $r >
0$ such that
\[ P \cap (\overline{B_r (y)} \setminus \{y\}) \subset D \, , \ \ \mbox{whereas} \ \
(\partial D \setminus \{y\}) \cap B_r (y) \subset \RR^2 \setminus P \, ,
\]
where $B_r (y) = \{ x \in \RR^2 : |x - y| < r \}$ (cf. \cite[Theorem 4.8]{Val}).

Then for $x \in B = B_{r/8} (x_0)$, where $x_0 = (0, r/4)$, we have that $d (x,
\partial D) > x_2$ provided $x_1 \neq 0$. Hence
\[ \int_B d (x, \partial D) \, \D x > \int_B x_2 \, \D x = \pi (r/8)^2 (r/4) =
\pi (r/8)^2 d (x_0, \partial D) \, ,
\]
and so the area mean-value inequality, guaranteeing the superharmonicity of $d
(\cdot, \partial D)$, is violated for~$B$.
\end{proof}

It should be emphasized that Theorem 1.3 is specifically a two-dimensional result,
because neither $D$ nor $\bar D$ need be convex in higher dimensions; in
\cite[Sect.~5]{AK}, Armitage and Kuran give an example confirming the latter
assertion.

%\vspace{-2mm}

{\centering \section{Proof of Theorem 1.1} }

\noindent It is sufficient to show that conditions \eqref{cond} imply that $d
(\cdot, \partial D)$ is superharmonic in $D$; indeed, this allows us to apply
Theorem 1.3, thus demonstrating the convexity of $D$.

For a nonvanishing $u \in C^2 (D)$, one obtains by a straightforward calculation
that
\[ \nabla^2 (\log u) = - \frac{|\nabla u|^2}{u^2} + \frac{\nabla^2 u}{u} \quad
\mbox{in} \ D.
\]
Therefore, if  $v (\cdot, \mu)$ is a solution of the Dirichlet problem \eqref{MHh},
\eqref{bc} with large $\mu > 0$, then it does not vanish, and so
\[ - \nabla^2 (\log v) = - \mu^2 + |\nabla v|^2 / v^2 < 0
\]
provided condition \eqref{cond} is valid. Hence, $- \mu^{-1} \log v (\cdot, \mu)$ is
superharmonic in $D$ for all large $\mu > 0$.

The next step is to demonstrate that passage to the limit in Theorem~1.2 preserves
superharmonicity. The proof is based on the following corollary of the Varadhan's
result \cite[Theorem~3.6]{V}.

\vspace{1.6mm}

\begin{proposition}
Let $v (x, \mu)$ be a solution of the Dirichlet problem \eqref{MHh}, \eqref{bc} in a
bounded domain $D \subset \RR^2$ with smooth boundary. Then for every $\rho \in (0,
1/2)$, there exists a constant $C_\rho > 1$ such that the estimate
\[ v (x, \mu) \leq C_\rho \exp \{ - \mu \, (1 - \rho) \, d (x, \partial D) \} \, ,
\quad x \in \bar D ,
\]
is valid for all $\mu > 0$.
\end{proposition}

In view of inequality \eqref{ineq} and the first condition \eqref{cond}, this
assertion implies that
\begin{equation}
- \mu^{-1} \log v (x, \mu) \geq - \mu^{-1} \log C_\rho + (1 - \rho) \, d (x,
\partial D) \, , \quad x \in \bar D . \label{final}
\end{equation}
Averaging this over an arbitrary open disc $B_r (x)$ such that $\overline{B_r (x)}
\subset D$, we obtain
\[ - \mu^{-1} \log v (x, \mu) \geq - \mu^{-1} \log C_\rho + \frac{1 - \rho}{\pi r^2}
\int_{B_r (x)} d (y, \partial D) \, \D y \, ,
\]
Indeed, the function on the left-hand side of \eqref{final} is superharmonic and the
first term on the right is harmonic. Letting $\mu \to \infty$ first and then $\rho
\to +0$, we see that the last inequality turns into
\[ d (x, \partial D) \geq \frac{1}{\pi r^2} \int_{B_r (x)} d (y, \partial D) \, \D y
\, , \quad x \in D ,
\]
in view of \eqref{lim}.

Since the obtained inequality is valid for all $r > 0$ such that $\overline{B_r (x)}
\subset D$, the distance function $d (\cdot, \partial D)$ is superharmonic in $D$.
Then Theorem~1.3 yields the assertion of Theorem~1.1, thus completing the proof.

\vspace{-12mm}

\renewcommand{\refname}{
\begin{center}{\Large\bf References}
\end{center}}
\makeatletter
\renewcommand{\@biblabel}[1]{#1.\hfill}
\makeatother


\begin{thebibliography}{9}


\fontsize{10.6pt}{-2mm}{\selectfont

\bibitem{AS} M. Abramowitz, I.~A. Stegun (eds.), {\it Handbook of Mathematical
Functions}. US National Bureau of Standards, Washington, DC, 1964.

\bibitem{AK} D.~H. Armitage, \"U. Kuran, ``The convexity of a domain and the
superharmonicity of the signed distance function,'' \textit{Proc. Amer. Math. Soc.}
{\bf 93}, 598--600 (1985).

\bibitem{B} R.~B. Burckel, ``Three secrets about harmonic functions,'' \textit{Amer.
Math. Monthly} {\bf 104}, 52--56 (1980).

\bibitem{D} R.~J. Duffin, ``Yukawan potential theory,'' \textit{J. Math. Anal.
Appl.} {\bf 35}, 105--130 (1971).

\bibitem{HMOSY} T. Hasebe, J. Masamune, T. Oka, K. Sakai, T. Yamada, ``Construction
of signed distance functions with an elliptic equation,'' preprint
arXiv:2401.17665v1

\bibitem{Val} F.~A. Valentine, \textit{Convex sets}, McGraw-Hill, New York, 1964.

\bibitem{V} S.~R.~S. Varadhan, ``On the behavior of the fundamental solution of the
heat equation with variable coefficients,'' \textit{Comm. Pure Appl. Math.} {\bf
20}, 431--455 (1967).

}
\end{thebibliography}
\end{document}